\documentclass{amsart}
\usepackage{graphicx,amsmath,amssymb} 
\usepackage[all,2cell]{xy}

\title{Equivariant geometry of low-dimensional quadrics}
\author{Brendan Hassett and Yuri Tschinkel}
\date{October 2024}

\theoremstyle{plain}
\newtheorem{prop}{Proposition}[section]

\newtheorem{theo}[prop]{Theorem}
\newtheorem{coro}[prop]{Corollary}

\newtheorem{lemm}[prop]{Lemma}

\theoremstyle{definition}
\newtheorem{defi}[prop]{Definition}

\newtheorem{rema}[prop]{Remark}

\newtheorem{exam}[prop]{Example}

\makeatother
\makeatletter

\newcommand{\bA}{\mathbb A}

\newcommand{\bF}{\mathbb F}
\newcommand{\bG}{\mathbb G}

\newcommand{\bP}{\mathbb P}

\newcommand{\cE}{\mathcal E}

\newcommand{\cO}{\mathcal O}

\newcommand{\cS}{\mathcal S}

\newcommand{\rH}{\mathrm H}

\newcommand{\fD}{\mathfrak D}
\newcommand{\fS}{\mathfrak S}

\newcommand{\sB}{\mathsf B}
\newcommand{\sG}{\mathsf G}
\newcommand{\sL}{\mathsf L}
\newcommand{\sP}{\mathsf P}
\newcommand{\sT}{\mathsf T}
\newcommand{\sU}{\mathsf U}

\newcommand{\Pf}{\mathrm{Pf}}

\newcommand{\Bl}{\operatorname{Bl}}

\newcommand{\Fl}{\operatorname{Fl}}

\newcommand{\Gr}{\operatorname{Gr}}

\def\GL{\mathsf{GL}}
\def\SL{\mathsf{SL}}
\newcommand{\PGL}{\mathsf{PGL}}

\newcommand{\SO}{\mathsf{SO}}

\newcommand{\Pic}{\operatorname{Pic}}
\newcommand{\Spec}{\operatorname{Spec}}

\newcommand{\ra}{\rightarrow}

\begin{document}

\begin{abstract}
We provide new stable linearizability constructions for regular actions 
of finite groups on homogeneous spaces and low-dimensional quadrics.
\end{abstract}

\maketitle

\section{Introduction}

Let $G$ be a finite group, acting generically freely and regularly on a smooth projective variety $X$. Of particular interest are linear actions, i.e., generically free actions on $\bP(V)$ arising from a linear representation $V$ of $G$. Actions which are equivariantly birational to a linear action are called {\em linearizable} (or {\em linear}); in particular $X$ is a rational variety. The classification of such actions, up to birationality, is an open problem even for linear actions on $\bP^2$ (see \cite{DI}, \cite{TYZ} and references therein). 

A related problem is to understand {\em stable} linearizability, i.e., linearizability of $G$-actions on $X\times \bP^m$, with trivial action on the second factor. 
Apart from its intrinsic interest, this 
property is relevant for the study of automorphisms of fields of invariants \cite{Kollar}.
Until recently, the only known instances of stable equivariant birationalities 
were those arising from faithful linear actions of $G$, and a particular nonlinearizable but stably linearizable action of the dihedral group $\fD_6$ (of order 12) on a quadric surface \cite{lemire}. New tools, such as the $G$-equivariant version of the universal torsor formalism of Colliot-Th\'el\`ene-Sansuc \cite{Colliot-S} allowed us to settle the stable linearizability problem for quadric surfaces \cite[Section 7]{HTNagoya}.

In this note, we focus on quadric hypersurfaces of dimension three
and four. This is an interesting class of examples; indeed, it is
already unknown whether or not every $\fS_3$-action on a smooth quadric threefold is linearizable. 
On the one hand, cohomological obstructions and the universal torsor formalism play a 
limited role for these varieties. On the other hand, this is a good testing ground
for stable linearization: the Pfaffian
constructions of \cite{BBT}, interpreted via coincidences among 
Lie groups; the reduction to $2$-Sylow subgroups, in the spirit of of Springer's Theorem, studied in
\cite{DR}; and techniques from the stable birational geometry 
of quadrics over nonclosed fields.  

In Section~\ref{sect:gen} we recall  basic notions from equivariant geometry and connect it to versality for group actions on varieties. 
In Section~\ref{sect:quad} we investigate the impact of isotropic subspaces on stable linearizability; 
see Theorem~\ref{theo:isotropicSL}. As in the case of quadrics over nonclosed fields \cite{Totaro}, the anisotropic quadrics have the 
richest birational geometry. In Section~\ref{sect:flag} we turn to flag varieties for special linear groups, and establish stable linearizability of translation actions in this context, in Theorem~\ref{thm:lin}. In Section~\ref{sect:quadric}, we provide Springer-type results, reducing stable linearization to $2$-Sylow subgroups. We interpret arguments for special orthogonal groups in terms of the geometry of Pfaffians. 
In Section~\ref{sect:quad3}, we present key examples of stable linearizable actions in dimension 3, and reduce the stable linearization problem to a specific action of the dihedral group $\fD_4$; see Corollary~\ref{coro:final}.

\ 

{\bf Acknowledgments:} 
The first author was partially supported by Simons Foundation Award 546235 and NSF grant 1929284,
the second author by NSF grant 2301983. 
We are grateful to Andrew Kresch for helpful
suggestions about this project.

\section{Groups actions, twists, and rationality}
\label{sect:gen}
Throughout, we work over a base field $k$ that is algebraically closed of characteristic zero.

\subsection{Notions of linearizability}

A $G$-variety is a smooth algebraic variety over $k$ with a generically free action of $G$. Here varieties are assumed to
be geometrically integral.  

We use the following four birational properties of a $G$-action on $X$:
\begin{itemize}
\item{{\em strictly linear} if there exists an equivariant birational map
$$V \stackrel{\sim}{\dashrightarrow} X,$$
where $V$ is a linear representation of $G$;}
\item{{\em linear} if there exists an 
equivariant birational map
$$\bP(V) \stackrel{\sim}{\dashrightarrow} X,$$
where $V$ is a linear representation of $G$;}
\item{{\em stably linear} if $X \times \bP^n$ is linear,
with trivial action on the second factor;}
\item{{\em projectively linear} if there
exists a projective representation of 
$G$ on $\bP(V)$ and an equivariant  birational map
$$
\bP(V) \stackrel{\sim}{\dashrightarrow} X.
$$}
\end{itemize}
In drawing analogies between equivariant geometry and geometry over nonclosed fields, one could view 
(strict, projective) linearizability as analogous to rationality, stable linearizability as analogous to stable rationality, etc.

\begin{rema} \label{rema:NoName}
The No-Name Lemma \cite[4.3]{CGR} implies that if $X \times W$ is
linear for some linear $G$-representation $W$ then 
$X$ is stably linear.
In particular, the notion of {\em strictly stably linear} is 
logically equivalent to the notion of {\em stably linear}.
\end{rema}

\subsection{Projective bundles}
\label{subsect:katsylo}

Our definition of linearizability requires generically
free actions. Without this assumption, the notion
behaves counter-intuitively:
\begin{exam}
Consider the dihedral group
$$G=\mathfrak{D}_4=\left< \sigma,\tau: \sigma^4=\tau^2=e,
\tau\sigma = \sigma^3\tau\right>$$
with representation $V$ given by 
$$\sigma = \left(\begin{matrix} i & 0 \\ 0 & i^3 \end{matrix}
\right),\quad \tau = \left( \begin{matrix} 0 & 1 \\ 1 & 0 \end{matrix}
\right).$$
This does {\em not} act generically freely on
$\bP(V)$ as $\sigma^2$ is trivial; the
quotient has nonzero Amitsur invariant \cite[Section~3.5]{HTNagoya}.
Consider the product $\bP(V) \times \bP^1$ with a trivial
action on the second factor; it also has nonzero Amitsur
invariant, as this is a
stable birational invariant. 
The product is not equivariantly
birational to $\bP(V')$ for any three-dimensional
representation of $G$, which would have trivial Amitsur
invariant.  
\end{exam}

Let $G$ be a finite group and $V$ a $G$-representation
such that the induced action on $\bP(V)$ is generically
free. The quotient map $V \dashrightarrow \bP(V)$ gives
a line bundle
$$\Bl_0(V) \rightarrow \bP(V),$$
so the No-Name Lemma implies $V$ and $\bP(V) \times \bA^1$
are equivariantly birational.  The inclusion $V\subset \bP(V\oplus 1)$ yields that $V$ and
$\bP(V) \times \bP^1$
are linearizable. Induction implies that  $\bP(V) \times \bP^n$ is
as well. 
In each case, $G$ acts trivially on the second factor.

We recall observations of \cite[\S 1]{Katsylo},
retaining the assumptions above:
\begin{itemize}
\item{Suppose $\cE \rightarrow 
\bP(V)$ is a vector bundle of rank $n+1$ with a lifting
of the $G$-action. Then 
the No-Name Lemma implies that 
$\cE$ is birational to $\cO_{\bP(V)}^{n+1}$
over $\bP(V)$, hence $\bP(\cE)$ is
birational to $\bP(V) \times \bP^n$.}
\item{Suppose $X$ is a smooth projective
variety with generically-free $G$-action
and $\cE \ra X$ a vector bundle with $G$-action.  
If $X$ is linearizable then $\bP(\cE)$ is also
linearizable.}
\end{itemize}
Here is an extension of these ideas:
\begin{prop} \label{prop:product}
Let $G$ and $H$ be finite groups
acting generically freely on
$\bP(V)$ and $\bP(W)$ respectively.
Then the induced action of 
$G \times H$ on $\bP(V) \times 
\bP(W)$ is stably linearizable.  
\end{prop}
\begin{proof}
Consider the dominant rational map
$$\bP(V \oplus W) \dashrightarrow
\bP(V) \times \bP(W)$$
which induces the morphism
$$\Bl_{\bP(V) \sqcup \bP(W)}\bP(V \oplus W)
\rightarrow \bP(V) \times \bP(W).$$
This is the projectivization of a rank-two
vector bundle $\cE \ra \bP(V) \times \bP(W)$ that is equivariant for
$G\times H$. 
The $G\times H$-action is generically
free on $\bP(V) \times \bP(W)$
hence $\cE$ is equivariantly birational 
to $\bP(V) \times \bP(W) \times \bA^2$,
with trivial action on the last factor.
Thus $\bP(V) \times \bP(W)$ is stably
linearizable.
\end{proof}
\begin{exam} \label{exam:nolinear}
This is the best possible; we cannot expect linearizability here.  
Fix odd integers $k,\ell$ and consider
the actions of $\mathfrak{D}_{k}$ 
and $C_{\ell}$ on $\bP(V)$ and $\bP(U)$,
where $V$ is the standard two-dimensional representation
of the dihedral group and $U$ has weights
$\zeta$ and $\zeta^{-1}$. 
Here $C_{\ell}=\left< \zeta\right>
\simeq \mu_{\ell}$. 
Keep in mind that
$$\bP(V)\times \bP(U) \hookrightarrow
\bP(V\otimes U)$$
as a quadric surface.
The classification of \cite{DI}
-- see the bottom of page 537 -- 
shows that $\bP(V) \times \bP(U)$
is not birational to $\bP(V')$
for any three-dimensional representation
of $\mathfrak{D}_{k} 
\times C_{\ell}$. 
\end{exam}

\subsection{Notions of versality} \label{subsect:NoV}
We recall the terminology of \cite{DR}, restricting to finite groups:
\begin{itemize}
    \item {\em weakly versal} {\bf (WV)}: for every field $K/k$ and every $G$-torsor $T\to \Spec(K)$ there is a $G$-equivariant $k$-morphism $T\to X$,
    \item {\em versal} {\bf (V)}: every $G$-invariant open $U\subset X$ is weakly versal,
    \item {\em very versal} {\bf (VV)}: there exists a linear representation $G\to \GL(V)$ and a dominant $G$-equivariant rational map $V\to X$,
    \item {\em stably linearizable} {\bf (SL)}: $X\times W$ is equivariantly birational to a linear representation of $G$,
    where $W$ is a linear space with trivial $G$-action. 
\end{itemize}

Note that in the definitions above:
\begin{itemize}
\item{$T$ is viewed as a $k$-scheme;}
\item{we could replace $W$ with a non-trivial linear representation
or the projectivization of a linear representation without
changing the definition.}
\end{itemize}

These notions are related:
$$
\mathbf{(SL)} \Rightarrow 
\mathbf{(VV)} \Rightarrow 
\mathbf{(V)} \Rightarrow 
\mathbf{(WV)}.
$$

\begin{rema}
Not every projectively linear action is very versal. 
Indeed, suppose we are given a projective faithful representation
$$\rho:G \rightarrow \PGL(V)$$
and an associated central extension
\begin{equation}
\label{eqn:seq}
1 \rightarrow \mu_n \rightarrow \widetilde{G} 
\rightarrow G \rightarrow 1
\end{equation}
admitting a linear representation
$$\widetilde{\rho}: \widetilde{G} \rightarrow \GL(V).
$$
By \cite[Proposition 9.1]{DR}, the $G$-action
on $\bP(V)$
is very versal  if and only if 
the exact sequence \eqref{eqn:seq} splits. 
\end{rema}

We recall the notion of a {\em twisting pair}, a tuple $(T,K)$ consisting of a field extension $K/k$ and a $G$-torsor $T$ over $K$; this gives rise to a twist ${}^TX$, the $K$-variety obtained by twisting via $T$. 

The connection between versality and rationality over nonclosed fields is addressed in 
\cite[Theorem 1.1]{DR}: 
\begin{itemize}
    \item {\bf (WV)}  $\Leftrightarrow$ 
${}^TX(K) \neq \emptyset$, for all twisting pairs $(T,K)$,
\item {\bf (V)}  $\Leftrightarrow$ 
${}^TX(K)$ are Zariski dense,
 for all $(T,K)$,
\item {\bf (VV)}  $\Leftrightarrow$ 
${}^TX$ is unirational over $K$,
 for all $(T,K)$,
\item {\bf (SL)}  $\Leftrightarrow$ 
${}^TX$ is stably rational over $K$,
 for all $(T,K)$.
\end{itemize}

\begin{rema}
An  analogous statement regarding rationality over $K$ fails: 
there are examples of stably linearizable but not linearizable quadric surfaces \cite{HTNagoya}. 
\end{rema} 

\section{Quadrics and isotropic subspaces}
\label{sect:quad}

Here we follow \cite[Section 2]{NP}. 
We continue to assume the base field
is algebraically closed of characteristic zero.

Let $(V,\mathsf{q})$ be a non-degenerate quadratic
form invariant under the action of 
a finite group $G$.

\begin{defi} 
An {\em isotropic subspace} $W \subseteq V$
is a $G$-invariant subspace such that
$\mathsf{q}|W=0$. We say $(V,\mathsf{q})$ is {\em anisotropic}
if it has no nonzero isotropic subspaces.
A {\em hyperbolic subspace} of $V$
is a pair of $G$-invariant isotropic subspaces $W, W^{\vee} \subset V$
such that 
$$\mathsf{q}(w,f)=f(w), \quad \text{ for all }
w\in W, f\in W^{\vee}.$$
In other words, the restriction of $\mathsf{q}$ to
$H_W:=W\oplus W^{\vee}$ has matrix
$$\left( \begin{matrix} 0 & I \\
            I & 0 \end{matrix} \right).$$
\end{defi}

Suppose an irreducible representation $W$ of $G$
admits a nonzero invariant quadratic form $\mathsf{q}$.
When $\mathsf{q}$ is non-degenerate, this induces
a $G$-equivariant isomorphism 
\begin{equation}
\label{eqn:ww}
W \stackrel{\sim}{\ra} W^{\vee},
\end{equation}
unique up to scalar.  
\begin{prop} \label{prop:gethyp}
Suppose $(V,\mathsf{q})$ is non-degenerate and
admits an irreducible isotropic subspace
$W$. Then we obtain 
$$(V,\mathsf{q})=(V',\mathsf{q}')\oplus_{\perp} H_W$$
where $(V',\mathsf{q}')$ is non-degenerate.
\end{prop}
\begin{proof}
Since $\mathsf{q}$ is non-degenerate, there
exists a copy of $W^{\vee} \subset V$
such that 
$$
\mathsf{q}|(W\oplus W^{\vee}) = 
\left( \begin{matrix} 0 & I \\
            I & R \end{matrix} \right)$$
where $R$ is a $G$-invariant quadratic
form on $W^{\vee}$. If $R=0$ then we have our hyperbolic form $H_W$. If $R\neq 0$ then, by \eqref{eqn:ww}, $R=\lambda I$
for some  $0 \neq \lambda \in k$. 
Since $k$ does not have characteristic two,
after equivariant row and column operations
$\mathsf{q}|W\oplus W^{\vee}$ becomes hyperbolic.
\end{proof}
Applying this inductively gives an equivariant version of Witt decomposition:
\begin{coro} \label{coro:anisotropic}
Every non-degenerate $(V,\mathsf{q})$ is equivariantly
equivalent to the orthogonal direct sum of a hyperbolic form and an anisotropic form.
\end{coro}

\begin{prop}
Let $X\subset \bP(W\oplus W^\vee)$ be a quadric hypersurface
associated with a hyperbolic form
$H_W$, where $W$ is a representation of $G$
of dimension $d\ge 2$. If $G$ acts generically freely on $\bP(W)$ 
then $X$ is linearizable. If $G$ acts generically
freely on $X$ then $X$ is stably linearizable.
\end{prop}
We already observed in Example~\ref{exam:nolinear} that we cannot
expect $X$ to be linearizable in general.

\begin{proof}
Consider the linear projection
from $\bP(W)$, which induces
$$\pi_{\bP(W)}:\Bl_{\bP(W)}X \ra \bP(W^{\vee}),
$$
which is a $\bP^{d-1}$ bundle -- indeed,
the projectivization of a vector bundle
with $G$-action. Note that $G$ acts generically
freely on $\bP(W)$ if and only if it acts generically
freely on $\bP(W^{\vee})$.
When $G$ acts generically freely on
$\bP(W^{\vee})$ then the observations of
Section~\ref{subsect:katsylo}
give that $X$ is linearizable.
Suppose that $G$ fails to acts generically
freely on $\bP(W)$.
Since it {\em does} act generically freely
on $X \subset \bP(W\oplus W^{\vee})$,
we conclude there is a cyclic
central subgroup
$$C_{\ell}= \left< \zeta \right> \subset G, \quad \ell \text{ odd, } \zeta^{\ell}=1,$$
acting via
\begin{align*}
C_{\ell} \times (W\oplus W^{\vee})
& \ra W \oplus W^{\vee} \\
\zeta\cdot(w,f) & \mapsto (\zeta w, \zeta^{-1}f).
\end{align*}
Let $U$ be the two-dimensional representation of $C_{\ell}$
with weights $\zeta$ and $\zeta^{-1}$. Take the basechange
of $\pi_{\bP(W)}$
$$(\Bl_{\bP(W)}X) \times \bP(U) \ra \bP(W^{\vee}) \times \bP(U)$$
which remains the projectivization of a $G$-equivariant
vector bundle; now the base
has a generically free action of $G$ on the base.
Since $\bP(W^{\vee}) \times \bP(U)$ is stably linearizable
by Proposition~\ref{prop:product}, $X \times \bP(U)$ is as well.
An application of the No-Name Lemma, as in
Section~\ref{subsect:katsylo}, implies that $X$ is
stably linearizable.
\end{proof}
The same projection argument yields
the following:
\begin{theo} \label{theo:isotropicSL}
Let $(V,\mathsf{q})$ be a non-degenerate quadratic
form invariant under the action of 
a finite group $G$. Let 
$$X=\{\mathsf{q}=0\} \subset \bP(V)$$
and assume that $\dim(X)>0$.
Let $0 \neq W\subset V$ be an
isotropic subspace for $\mathsf{q}$ 
and 
$$H_W\simeq W \oplus W^{\vee} \subseteq V$$
the hyperbolic subspace 
guaranteed by Proposition~\ref{prop:gethyp}.
If $G$ acts generically freely on 
$\bP(W^{\vee} \oplus H_W^{\perp})$ 
then $X$ is linearizable. If $G$ acts generically
freely on $X$ then $X$ is stably linearizable.
\end{theo}
When $\dim(W)=1$ the projection is birational, and we obtain the well-known
\begin{coro} \label{coro:fix}
Retain the notation of
Theorem~\ref{theo:isotropicSL}.
If $X$ has a fixed point then $X$ is linearizable.
\end{coro}

In light of Theorem~\ref{theo:isotropicSL}
and Corollary~\ref{coro:anisotropic},
the natural question for future study
is the stable birational classification
of anisotropic quadrics, i.e.,
$$X=\{\mathsf{q}=0 \} \subset \bP(V),
$$
where 
\begin{equation}\label{eq:orth}
(V,\mathsf{q}) = (V_1,\mathsf{q}_1) 
\oplus_{\perp} \cdots \oplus_{\perp}
(V_r,\mathsf{q}_r)
\end{equation}
is an orthogonal direct sum of 
self-dual irreducible representations
of $G$ with their distinguished quadratic forms. 

There are only finitely many possibilities
to consider, thanks to our next result:
\begin{prop} \label{prop:mult}
Suppose that $\mathsf{q}$ is direct
sum of irreducible non-degenerate 
quadratic forms as in (\ref{eq:orth}). 
If $\mathsf{q}$ is anisotropic then the 
factors $(V_i,\mathsf{q}_i)$ are not isomorphic.
\end{prop}
\begin{proof}
Suppose a summand, say $(V_1,\mathsf{q}_1)$,
appears with multiplicity. After
rescaling if necessary, we obtain
$$(V_1,\mathsf{q}_1) \oplus_{\perp}
(V_1,\mathsf{q}_1) \subset (V,\mathsf{q}).$$
This contains an isotropic subspace -- the image
of $V_1$ under $v \mapsto (v,iv)$ --
contradicting the assumption that $(V,\mathsf{q})$ is anisotropic.
\end{proof}

\section{Flag varieties and special groups}
\label{sect:flag}

Here we consider natural actions of finite groups on homogeneous
spaces for special classes of algebraic groups, with a view
toward stable linearizability.
The most fundamental construction is the No-Name Lemma mentioned
in Remark~\ref{rema:NoName}.  We will apply it freely as we
present further applications below. 

An algebraic group $\sG$ over $k$ is called {\em special} if
$$
\rH^1(K,\sG) =0, \quad \forall K/k.
$$
Examples (listed in \cite[Section 4.2]{CT-inv}) include
\begin{itemize}
\item split connected solvable groups,
\item $\GL_n$, $\SL_n$, $\mathsf{Sp}_{2n}$, split $\mathsf{Spin}_n$, with $n\le 6$.
\end{itemize}
By \cite[p. 26]{Serre1958}, the only special semisimple groups 
over $k$ are products of $\mathsf{SL}_n$ and $\mathsf{Sp}_n$. 

\begin{prop}
Let $\sG$ be a special connected linear algebraic group and $G\subset \sG$ a finite subgroup. 
Then the translation action of $G$ on $\sG$ is stably linearizable.
\end{prop}

This generalizes \cite[Proposition 3.7]{BT-unram}.

\begin{proof}
We are following the strategy of \cite[Prop.~4.9]{CT-inv}.  Choose a representation
$G \hookrightarrow \GL_n$ for some $n$. Note that $\sG$ is rational over $k$.  Consider the diagonal
embedding
$$
G\hookrightarrow \GL_n \times \sG
$$
and the projection onto $\sG$. We claim that this is $G$-birational to
$\GL_n \times \sG$, with trivial action on the first factor.  
Indeed, a finite group action on $\GL_n$ is tautologically linearizable
over any field.  Since $G$ acts generically freely on $\sG$, the No-Name
Lemma implies the desired birational map.  

On the other hand, consider the projection 
$$\pi_1: \GL_n \times \sG \rightarrow \GL_n.$$ 
This is a torsor
for $\sG$ in the sense that it admits a section $s$ after basechange
$$G \times \GL_n \rightarrow  \GL_n,$$
namely, 
$$s(\gamma,g) \mapsto (\gamma \cdot 1_{\sG}, \gamma \cdot g).$$
However, the speciality assumption implies that there is a 
section even over the function field of $\GL_n$. Thus
$\sG \times \GL_n$ with the diagonal action of $G$
is equivariantly birational to $\sG \times \GL_n$, 
with trivial action on the first factor. 

Putting this together, we find that $\sG$ and $\GL_n$ are stably 
birational as $G$-varieties. It follows that $\sG$ is stably
linearizable.  
\end{proof}

\begin{coro}
Let $\sG$ be special and $\sU \subset \sG$ a unipotent subgroup.
Fix a finite subgroup $G \subset \sG$. Then the induced left action on 
$\sG/\sU$ is stably linearizable.
\end{coro}
\begin{proof}
Since $\sU$ is unipotent and the characteristic is zero, $G\cap \sU = \{1\}$
and $G$ acts generically freely on $\sG/\sU$.  
Consider the projection
$$\sG \rightarrow \sG/\sU,$$
which is a vector bundle as the group $\sU$ is special \cite[Section~4.2]{CT-inv}.
The No-Name Lemma implies $\sG$ is $G$-birational to the product
$$\sU \times (\sG/\sU)$$
with $G$ acting trivially on the first factor. 
Since the $G$-action on $\sG$ is stably linearizable the same holds for the quotient.  
\end{proof}

\begin{rema}
\label{rema:sometimes}
Even when $\sG$ is not special, one can sometimes establish the linearizability of the translation action. For example, $\sG=\mathsf{PGL}_2$ admits an equivariant compactification to $\bP^3$, and the translation action extends as a projectively linear action on $\bP^3$. 
The obstruction to linearizability of this action is captured by the Amitsur invariant (see, e.g.~\cite[Section~3.5]{HTNagoya}).
\end{rema}

\begin{prop} \label{prop:flag}
Let $\sG$ be special and $\sB \subset \sG$ a Borel subgroup. Let 
$G\subset \sG$ be a finite subgroup such that the induced left action on the quotient $\sG/\sB$ is generically free. Then this action 
is stably linearizable.
\end{prop}
\begin{proof}
Express $\sB = \sU \rtimes \sT$ where $\sT$ is a maximal torus which gives
$$\varpi:\sG/\sU \rightarrow \sG/\sB=:\Fl$$
with fibers isomorphic to $\sT$ Indeed, if $L_1,\ldots,L_r$ are line bundles
forming a basis for $\Pic(\Fl)$ then we may interpret
$$\sG/\sU = L^*_1 \times_{\Fl} \cdots \times_{\Fl} L^*_r,
\quad L^*_i = L_i \setminus 0;$$
in particular, the action of our finite group linearizes to each of the $L_i$.
Our generic-freeness assumption means that $\sG/\sU$ is equivariantly
birational to $\bA^r \times \Fl$ with trivial action on the first factor.
Thus $\sG/\sB$ is stably linearizable.
\end{proof}

We record an application of the Serre-Grothendieck classification
of special groups \cite[Section~4]{Serre1958}:

\begin{lemm} \label{lemm:serre}
Let $\sG$ denote a special semisimple linear algebraic group 
and $\sP \subset \sG$ a (split) parabolic subgroup. 
Then the Levi factors of $\sP$ are also special.  
\end{lemm}

\begin{theo}
\label{thm:lin}
Let $\sG$ denote a special semisimple group, 
$\sP \subset \sG$ a split parabolic subgroup, and $G \hookrightarrow \sG$
a finite group. If $G$ acts generically freely on $\sG/\sP$ then
this action is stably linearizable.
\end{theo}
\begin{proof} 
We proceed by induction on the dimension of $\sG$. 

Choose a maximal sequence of parabolic subgroups
$$\sP =: \sP_1 \subsetneq \sP_2 \subsetneq \cdots \sP_{m-1} \subsetneq \sP_m=\sB$$
and consider the tower
$$\sG/\sB \rightarrow \sG/\sP_{n-1} \rightarrow \cdots \rightarrow \sG/\sP.$$
Each step of this tower is fibered in 
projective homogeneous spaces for $\sL$, a Levi subgroup of some parabolic 
subgroup of $\sG$. These groups are special by Lemma~\ref{lemm:serre}, 
and their homogeneous spaces are stably linearizable  by induction. 
As before, iterating the No-Name Lemma yields
$$\sG/\sB \,\, \text{stab.~lin.} \Rightarrow \sG/\sP_{n-1}\,\,  \text{stab.~lin.}
\Rightarrow \cdots \Rightarrow \sG/\sP \,\, \text{stab.~lin.}$$
\end{proof}

\section{Quadric hypersurfaces}
\label{sect:quadric}

\subsection{Arbitrary dimensions}
Before considering specific actions of finite groups on quadrics, we record general considerations regarding the presentation of such actions. Let 
$$
X\subset \bP^n
$$
be a smooth quadric hypersurface. Let $\mathsf q$ be the associated quadratic form, which is unique, up to scalars. We consider finite subgroups $G\subset \PGL_{n+1}$
preserving $X$. 

The Amitsur invariant \cite[Section 3.5]{HTNagoya} yields:
\begin{prop}
If the action of $G$ on $X$ is stably linearizable then we can lift
$G \subset \GL_{n+1}$.
\end{prop}

When the number of variables is odd, we can always work with 
the special orthogonal group:
\begin{prop} \label{prop:norm}
If $n=2m$ is even then we may assume $G \subset \SL_{2m+1}$.
\end{prop}
\begin{proof}
Given 
$$
\varrho: G \hookrightarrow \mathsf{O}_{2m+1}
$$
there is a modified representation
$$\rho = \det(\varrho) \cdot \varrho: G \rightarrow \SO_{2m+1}
$$
that is projectively equivalent to $\varrho$. 
Note that $\rho$ is injective if and only if 
the image of $\varrho$ does not contain $-I$, 
i.e. $\varrho$ acts generically freely on $\bP^{2m}$.  
\end{proof}

\subsection{Pfaffian constructions}
\label{subsect:Pfaff}
We recall the Pfaffian construction,
for quadric hypersurfaces; see 
\cite[Section 7]{BBT} for further details.

Let $M$ denote an antisymmetric 
$2r\times 2r$ matrix. The 
{\em Pfaffian form} $\Pf(M)$ 
is a homogeneous form of degree $r$
such that
$$\Pf(M)^2 = \det(M).$$
When $r=2$
$$M=\left( \begin{matrix}
0 & m_{12} & m_{13} & m_{14} \\
-m_{12} & 0 & m_{23} & m_{24} \\
-m_{13} & -m_{23} & 0 & m_{34} \\
-m_{14} & -m_{24} & -m_{34} & 0
\end{matrix}
\right)$$
we have
$$\Pf(M) = m_{12}m_{34} - m_{13}m_{23} + m_{14}m_{23}$$
the Pl\"ucker relation for the Grassmannian 
$$\Gr(2,W) \subset \bP(\wedge^2 W),
\quad \dim(W)=4.$$
We recall a few properties:
\begin{itemize}
\item{Let $(V,\mathsf{q})$ be a non-degenerate
quadratic form where $V$ is a 
six-dimensional representation of $G$.
This is Pfaffian if and only if
$$V \simeq \wedge^2 W, \quad \dim(W)=4$$
and $\mathsf{q}$ coincides with the (symmetric!) wedge pairing
$$\wedge^2 W \times \wedge^2 W 
\ra \wedge^4 W.$$}
\item{Let $V$ be a non-degenerate
quadratic form with $\dim(V)=5$. 
Then $V$ is Pfaffian if and only if
there exists a symplectic representation
$$(W,\omega), \quad \dim(W)=4$$ 
of $G$, with 
$$\wedge^2 W  = \left<\omega\right>
\oplus V.$$}
\item{In either case, assuming
$G$ acts generically freely on
$$X=\{\mathsf{q}=0\} \subset \bP(V),$$ 
then $X$ is stably birational to $W$.}
\end{itemize}
The last statement arises from the 
inclusion and projection morphisms
$$\cS|X \hookrightarrow W\times X
\rightarrow W,$$
where $\cS\ra \Gr(2,W)$ is the universal
subbundle.


\subsection{Springer-type results}
The following is a corollary of \cite[Theorem~10.2]{DR}:
\begin{prop} \label{prop:springer}
Let $X \subset \bP^n,n\ge 2,$ denote a smooth quadric hypersurface,
with a generically-free action by a finite group $G$. 
Let $G_2 \subseteq G$ denote a $2$-Sylow subgroup. 
Then the $G$-action on $X$ is stably linearizable if and only
if the $G_2$-action on $X$ is stably linearizable.
\end{prop}
Recall the equivalences stated in Section~\ref{subsect:NoV}:
In particular, a group action on $X$ is weakly versal if each 
twist by the group admits rational points and stably linearizable if
each twist is stably rational.  However, Springer's Theorem
and stereographic projection
guarantee that, for a smooth positive-dimensional quadric hypersurface $Y$ over a field $L$,
the following are equivalent:
\begin{itemize}
\item{$Y$ is rational over $L$;}
\item{$Y$ is stably rational over $L$;}
\item{$Y$ has a rational point over $L$;}
\item{$Y$ has a rational point over an odd-degree extension
of $L$.}
\end{itemize}
Now the last condition corresponds to passing from $G$
to a $2$-Sylow subgroup $G_2$, thus Proposition~\ref{prop:springer}
follows.

\begin{rema}
This argument is a bit vexing, as we are {\em not} showing that
linearizability can be checked on passage to a $2$-Sylow subgroup!
The dictionary of Section~\ref{subsect:NoV} leaves out
birational equivalence to a $G$-representation or its 
projectivization.  
\end{rema}
\begin{exam}
Consider the action of $G=\fS_3\times C_2$ on 
$$X=\{x_1^2 + x_2^2 + x_3^2+x_4^2 =0 \} \subset \bP^3$$
where $C_2$ acts via $x_4 \mapsto -x_4$ and $\fS_3$ 
permutes $x_1,x_2,$ and $x_3$.  By \cite{Isk-s3} this action
is not linearizable; a proof using Burnside invariants
may be found in \cite[Section~7.6]{HKTsmall}.  
On the other hand, the $2$-Sylow subgroup 
$G_2$ fixing $x_3$ has fixed points $\{x_4=x_1-x_2=0\} \cap X$.
The stable linearizability of this action has been shown, using different techniques, in
\cite{lemire} and \cite[Section~6]{HTNagoya}.
\end{exam}

\subsection{Quadric threefolds} 
\label{subsect:3-pfaff}
Theorem~\ref{thm:lin} -- combined with the symplectic
interpretation of $\mathsf{Spin}_5$ and the fact
that symplectic groups are special -- yields the following:

\begin{theo} \label{theo:threefold}
Let $X \subset \bP^4$ be a smooth three-dimensional quadric.
Suppose that $G \subset \SO_5$ acts on $X$ generically freely.
The action of $G$ is stably linearizable if the induced extension
$$
\centerline{
\xymatrix{
 1  \ar[r] &  \mu_2 \ar[r] \ar@{=}[d] &  \widetilde{G} \ar@{^{(}->}[d]\ar[r]  & G \ar[r] \ar@{^{(}->}[d]& 1 \\
  1 \ar[r] &  \mu_2 \ar[r] & \mathsf{Sp}_4 \ar[r]  & \mathsf{SO}_5 \ar[r]& 1 
}
}
$$
is trivial, i.e., the restriction homomorphism
$$\rH^2(\SO_5,\mu_2) \rightarrow \rH^2(G,\mu_2)$$
vanishes on the distinguished extension.  
\end{theo}

Here is a geometric interpretation of this theorem: 
An action of $G \subset \SO_5$ on $X\subset \bP^4$ lifts to the
spin group if and only if there is a $G$-equivariant imbedding
$$X \hookrightarrow \Gr(2,4)$$
arising from a representation $G \rightarrow \SL_4$ leaving 
a non-degenerate $2$-form invariant. 
This is the Pfaffian construction
from Section~\ref{subsect:Pfaff}.

\begin{exam}
The converse of Theorem~\ref{theo:threefold} is not true: 
There are linearizable quadric threefolds $X\subset \bP^4$
such that $G \subset \SO_5$ does not lift to $\mathsf{Sp}_4$.
In geometric terms, the variety of lines $F_1(X)$ -- a four-dimensional
projective representation of $G$ -- may have non-vanishing 
invariant.  

Here is a construction: Let $C$ denote a conic with non-trivial Amitsur
invariant, corresponding to a projective representation
$$\phi: G \rightarrow \PGL_2$$
not lifting to a linear representation.  Let $V$ denote the linear
representation associated with the symmetric square of $\phi$; there is
an embedding 
$$C \hookrightarrow \bP(V) \subset \bP( \mathbf{1} \oplus V).$$
The blowup of this projective space along $C$ admits a morphism
$$\varpi: \mathrm{Bl}_C(\bP(\mathbf{1} \oplus V)) \stackrel{\sim}{\longrightarrow} X 
\subset \bP^4$$
given by the linear system of quadrics vanishing along $C$.
Write 
$$x = \varpi(\bP(V)) \in X$$
for the image of the proper transform of the plane spanned by $C$.  
Consider the lines in $\bP(\mathbf{1} \oplus V)$ meeting $C$ at a point,
a projective plane bundle $W\rightarrow C$. (Secants to $C$ are counted twice!)
The No-Name Lemma implies that 
$$W \stackrel{\sim}{\dashrightarrow} \bP^2 \times C,
$$
where the first factor has trivial $G$-action.
The morphism $\varpi$ induces
$$\pi:W \stackrel{\sim}{\longrightarrow} F_1(X)$$
that blows up the lines 
$$\{ \ell : x \in \ell \subset X\} \simeq C.$$
We conclude that $F_1(X)$ -- birational to $C \times \bP^2$ -- also
has non-trivial Amitsur invariant.  
\end{exam}

\subsection{Quadric fourfolds}

Similarly, we observe:
\begin{theo}
\label{thm:4}
Let $X \subset \bP^5$ be a smooth four-dimensional quadric.
Suppose that $G \subset \SO_6$ acts on $X$ generically freely.
The action of $G$ is stably linearizable if the induced extension
$$
\centerline{
\xymatrix{
 1  \ar[r] &  \mu_2 \ar[r] \ar@{=}[d] &  \widetilde{G} \ar@{^{(}->}[d]\ar[r]  & G \ar[r] \ar@{^{(}->}[d]& 1 \\
  1 \ar[r] &  \mu_2 \ar[r] & \mathsf{SL}_4 \ar[r]  & \mathsf{SO}_6 \ar[r]& 1 
}
}
$$
is trivial, i.e., the restriction homomorphism
$$\rH^2(\SO_6,\mu_2) \rightarrow \rH^2(G,\mu_2)$$
vanishes on the distinguished extension.  
\end{theo}
Compare this with the Pfaffian interpretation in  Section~\ref{subsect:Pfaff}.

\begin{exam}
Let $V$ be a 4-dimensional representation of $\fS_5$. Its exterior square is the 6-dimensional representation; and there is a unique invariant quadric $X\subset \bP^5$.
The $\fS_5$-action on $X$ is not known to be linearizable. Theorem~\ref{thm:4},  combined with the Pfaffian construction, yields stable linearizability for the action on $X$. 
\end{exam}

\section{Applications to threefolds}
\label{sect:quad3}

In this section, we present examples of stable linearizability
constructions, focusing on cases where linearizability is not known.

We let $X$ be a smooth quadric threefold with a generically free regular action of a finite group $G$. We recall a ``nonstandard" linearizability construction, see, e.g., \cite[Sect.~5.8]{CalabiFano}: 
The infinite dihedral group $\bG_m \rtimes \mu_2$ 
acts on the quintic del Pezzo threefold $V_5 \subset \bP^6$
-- which has automorphism group $\PGL_2$. The action lifts
to a linear representation in $\GL_7$; and it stabilizes
\begin{itemize}
\item{a twisted cubic curve $R_3 \subset V_5$;}
\item{a conic $R_2 \subset V_5$;}
\item{a line $R_1 \subset V_5$,}
\end{itemize}
see \cite[Cor.~5.39]{CalabiFano}:
Projection from $R_2$ gives an equivariant birational map
$$
V_5 \stackrel{\sim}{\dashrightarrow} \bP^3.
$$
Projection from $R_1$ gives an equivariant birational map
$$\pi_{C_1}:V_5 \stackrel{\sim}{\dashrightarrow} X$$
onto a smooth quadric threefold. 

The action on 
$$X=\{x_1x_2=x_3x_4+x_5^2\}\subset \bP^4_{x_1,\ldots,x_5}$$ 
is given by 
$$
\begin{array}{rcl}
\tau:   (x_1,x_2,x_3,x_4,x_5)& \mapsto & (x_2,x_1,x_4,x_3,x_5) \\
\sigma: (x_1,x_2,x_3,x_4,x_5) & \mapsto & (\lambda^{-3}x_1, \lambda^3x_2, 
\lambda^{-1}x_3,\lambda x_4,x_5),  
\end{array}
$$
where $\lambda$ is a character of $\bG_m$.

For example, setting $\lambda=e^{2\pi i/3}$ gives a linearization of the $\fS_3$-action 
$$
X\subset \bP(V), \quad V=V_{\pm}\oplus V_2\oplus \mathbf 1, 
$$
where $V_{\pm}$ is the permutation representation on $x_1,x_2$,  
and $V_2$ is the unique faithful 2-dimensional representation of $\fS_3$. To our knowledge, linearizability of other actions of $\fS_3$ is unknown. For $\fD_8$ ($\lambda = e^{\pi i/4}$) we obtain linearizability when 
$$
X\subset \bP(V)\quad V=V_2\oplus V_2'\oplus \mathbf 1, 
$$
where $V_2$ and $V_2'$ are {\em different} faithful 2-dimensional representations of $\fD_8$.

\begin{prop}
\label{prop:spr}
Let
$$
X:=\{ x_1x_2=x_3x_4 +x_5^2\}
$$
and $G=\fD_{4n}$ with $n$ odd, acting via faithful 2-dimensional representations on $\{x_1,x_2\}$ and $\{x_3,x_4\}$, and trivially on $x_5$. Then the $G$-action on $X$ is stably linearizable. 
\end{prop}

\begin{proof}
Proposition~\ref{prop:springer} reduces
us to the $2$-Sylow subgroup $\fD_4 \subset \fD_{2n}$, which 
acts via
$$
\begin{array}{rcl}
\tau:   (x_1,x_2,x_3,x_4,x_5)& \mapsto & (x_2,x_1,x_4,x_3,x_5) \\
\sigma: (x_1,x_2,x_3,x_4,x_5) & \mapsto & (i x_1, i^3 x_2, 
\iota^3 x_3, \iota x_4,x_5),  
\end{array}
$$
where $i$ and $\iota$ are primitive fourth roots of unity.  
This has an isotropic subspace by Proposition~\ref{prop:mult} and hence is
linearizable (for $\fD_4$) by Theorem~\ref{theo:isotropicSL}.
\end{proof}

We now turn to the diagonal quadric
\begin{equation}
\label{eqn:dia}
X=\{ x_1^2+x_2^2+x_3^2+x_4^2 +x_5^2=0\}\subset \bP^4,
\end{equation}
with an action of a subgroup $G\subset W(\mathsf D_5)$, the Weyl group of $\mathsf D_5$, via signed permutations. 
In \cite[Section 9]{TYZ}, there is a classification of $G$ such that 
\begin{itemize}
\item all abelian subgroups $H\subset G$ have fixed points,
\item $G$ does not have a fixed point. 
\end{itemize}
Recall that existence of fixed points for generically free actions of abelian groups is a birational invariant of smooth projective varieties and that linear actions of abelian groups have fixed points. Furthermore, a fixed point yields a linearization for the action. 

The maximal groups on the list in \cite{TYZ} are:
$$
\fS_5, \quad  \fS_4, \quad  C_4\wr C_2,\quad  \mathsf{GL}_2(\bF_3), \quad \fS_3\times\fD_4, \quad \fD_8,
$$
and the maximal $2$-Sylow subgroups without fixed points are
$$
\fD_4\subset \fS_4, \quad C_4\wr C_2, \quad SD16 \subset  \mathsf{GL}_2(\bF_3), \quad \fD_8,  
$$
with specific 5-dimensional representations $V$, giving rise to $X\subset \bP(V)$.

\begin{prop}
\label{prop:another}
The $G$-action on $X$ in \eqref{eqn:dia} 
is stably linearizable if:
\begin{itemize}  
\item[(1)] 
$G=\fS_5$, 
with the standard permutation action, 
\item[(2)] 
$G=\fS_3\times C_2^2\subset \fS_3\times \fD_4$, 
with the standard permutation action of $\fS_3$ on the first variables and sign changes on $x_4,x_5$. 
\end{itemize}
\end{prop}

\begin{proof}
This is another corollary of Proposition~\ref{prop:springer}:
the $2$-Sylow subgroup 
$G_2\subset G$ has fixed points. 
\end{proof}

\begin{rema}
The $G$-action in Case (1) is known to admit only two Mori fiber space models, and in particular is not linearizable, 
by \cite[Theorem 3.1]{cheltsov}; in Case (2), the Burnside obstructions of \cite{BnG} prevent linearizability, see \cite[Example 9.2]{TYZ}.
\end{rema}

Theorem~\ref{theo:isotropicSL} yields the stronger result: 

\begin{prop}
\label{prop:iso-app}
The following $G$-actions on the quadric
$$
X:=\{ x_1^2+x_2^2+x_3^2+x_4^2 +x_5^2=0\}\subset \bP^4, 
$$
are linearizable:
\begin{itemize}
\item the unique (up to conjugation) $G=C_4\wr C_2\subset W(\mathsf D_5)$,
\item the unique semidihedral group $SD16\subset W(\mathsf D_5)$, 
\item the unique dihedral group $\mathfrak D_8\subset  W(\mathsf D_5)$. 
\end{itemize}
\end{prop}

\begin{proof}
The restriction of the $W(\mathsf D_5)$-action to the (unique) $C_4\wr C_2$ has character 
$
( 5, -3, 1, 1, 1, 1, 1, -3, 1, 1, -3, 1, -1, -1 ). 
$
Concretely, the action is given by
{\small
$$
\sigma_1:=\begin{pmatrix}
0 &   0 &    1 &  0\\
0 &   0 &  0 &  1\\
1 &0 &  0 &  0\\
 0 & 1 &   0 & 0
 \end{pmatrix}, \quad 
\sigma_2:=\begin{pmatrix}
0 &   -1 &    0 &  0\\
1 &   0 &  0 & 0\\
0 &0 &  1 &  0\\
 0 & 0 &  0 & 1
\end{pmatrix}, \quad
\sigma_3:=
\begin{pmatrix}
1 &   0 &    0 & 0\\
0 &   1 &  0 &  0\\
 0 & 0 &  0 &  -1\\
0 & 0 &   1 & 0
\end{pmatrix},
$$
}
$$
\sigma_4:=\mathrm{diag}(-1,-1,1,1), \quad \sigma_5:=
\mathrm{diag}(1,1,-1,-1), 
$$
with $\sigma_i$ acting on the variables $x_1,\ldots, x_4$ as indicated, with $\sigma_1: x_5\mapsto -x_5$, and all other generators acting trivially on $x_5$. 
In particular, 
$$
X\subset \bP^4=\bP(V_2\oplus V_2'\oplus \chi),
$$
where the characters are given by
$$
\begin{array}{rcl}
\mathrm{char}(V_2)&=&(2, -2,  0,  0,  2i, -2i, 1-i,  0, 1+i,-1-i,-1+i,  0,  0,  0),\\
\mathrm{char}(V_2')&=&(2, -2,  0,  0,  -2i, 2i, 1+i,  0, 1-i,-1+i,-1-i,  0,  0,  0).
\end{array}
$$
Theorem~\ref{theo:isotropicSL} applies.

We turn to $SD16$: the restriction of the $W(\mathsf D_5)$-action
to $G=SD16$ yields a representation with character
$
( 5, -3, -1, 1, 1, -1, -1 ).
$
The action is given by 
{\small
$$
\sigma_1:=\begin{pmatrix}
1 &   0 &    0 &  0\\
0 &   0 &  -1 &  0\\
 0 &-1 &  0 &  0\\
 0 & 0 &   0 & -1
 \end{pmatrix}, \quad 
\sigma_2:=\begin{pmatrix}
0 &   1 &    0 &  0\\
0 &   0 &  0 &  -1\\
 1 &0 &  0 &  0\\
 0 & 0 &   1 & 0
\end{pmatrix}, \quad
\sigma_3:=\begin{pmatrix}
0 &   0 &    0 &  -1\\
0 &   0 &  -1 &  0\\
 0 & 1 &  0 &  0\\
 1 & 0 &   0 & 0
 \end{pmatrix},
$$
}
$$
\sigma_4:=\mathrm{diag}(-1,-1,-1,-1), 
$$
with $\sigma_i$ acting on $x_1,\ldots, x_4$ as indicated and 
$\sigma_1,\sigma_2$ acting via $x_5\mapsto -x_5$, and $\sigma_3,\sigma_4$ acting trivially on $x_5$. Thus
$$
X \subset \bP^4=\bP(V_2\oplus V_2'\oplus \chi)
$$
as an invariant quadric with 
generically free $G$-action, where $V_2,V_2'$ are 2-dimensional faithful, complex conjugate, representations of $G$, and $\chi$ is a character. Since $V_2'=V_2^\vee$, Theorem~\ref{theo:isotropicSL} applies.

We repeat the analysis for $G=\mathfrak D_8$. The restriction of the $W(\mathsf D_5)$-action to $G$ yields a representation $V$ with character
$
( 5, -3, -1, 1, 1, -1, -1 ),   
$
which decomposes as 
$$
V=V_2\oplus V_2' \oplus \chi, 
$$
where the characters are given by
$$
\begin{array}{rcl}
\mathrm{char}(V_2) & = & (2,-2,0,0,0, \sqrt{2}, -\sqrt{2}) \\
\mathrm{char}(V_2') & = & (2,-2,0,0,0, -\sqrt{2}, \sqrt{2}) \\
\chi & = & (1,  1, -1,  1,  1,  -1,  -1),
\end{array}
$$
and the standard linearizability via a twisted cubic applies.
\end{proof}

\begin{coro}
\label{coro:final}
Let $G\subset W(\mathsf D_5)$ be such that the induced action on the diagonal quadric 
$$
X\subset \bP^4=\bP(V),
$$
via the standard irreducible 5-dimensional representation 
$V$ of $W(\mathsf D_5)$
satisfies the following properties:
\begin{itemize}
\item for every abelian $H\subseteq G$ one has $X^H\neq\emptyset$,
\item $G$ does not contain a subgroup $H\simeq \mathfrak D_4$ such that the restriction
of the representation $V$ to $H$ decomposes as
$$
V|_H = V_2\oplus \chi\oplus\chi'\oplus \chi'',
$$
where $\chi,\chi',\chi''$ are pairwise distinct characters of $H$. 
\end{itemize}
Then the $G$-action on $X$ is stably linearizable. 
\end{coro}


\begin{rema}
The methods above give no information about (stable) linearizability of the following $G$-actions on smooth quadric threefolds $X\subset \bP(V)$: 
\begin{itemize}
\item $G=\mathfrak D_4$ and $V=V_2\oplus \chi\oplus \chi'\oplus \chi''$, 
where $V_2$ is the unique irreducible 2-dimensional representation and $\chi,\chi',\chi''$ are pairwise distinct characters. 
\item $G=\mathfrak D_8$ and $V=V_2\oplus V_2' \oplus \mathbf 1$, where $V_2$ is the non-faithful 2-dimensional irreducible representation and $V_2'$ is a faithful representation of $G$. 
\end{itemize}
\end{rema}

\bibliographystyle{alpha}
\bibliography{EGHS}

\newcommand{\etalchar}[1]{$^{#1}$}
\begin{thebibliography}{ACC{\etalchar{+}}23}

\bibitem[ACC{\etalchar{+}}23]{CalabiFano}
C.~Araujo, A.-M. Castravet, I.~Cheltsov, K.~Fujita, A.-S. Kaloghiros,
  J.~Martinez-Garcia, C.~Shramov, H.~S{\"u}{\ss}, and N.~Viswanathan.
\newblock {\em The {Calabi} problem for {Fano} threefolds}, volume 485 of {\em
  Lond. Math. Soc. Lect. Note Ser.}
\newblock Cambridge: Cambridge University Press, 2023.

\bibitem[BPT10]{BT-unram}
F.~Bogomolov, T.~Petrov, and Yu. Tschinkel.
\newblock Unramified cohomology of finite groups of {L}ie type.
\newblock In {\em Cohomological and geometric approaches to rationality
  problems}, volume 282 of {\em Progr. Math.}, pages 55--73. Birkh\"{a}user
  Boston, Boston, MA, 2010.

\bibitem[BvBT23]{BBT}
Chr. Böhning, H.-Chr.~Graf von Bothmer, and Yu. Tschinkel.
\newblock Equivariant birational geometry of cubic fourfolds and derived
  categories, 2023.
\newblock {\tt arxiv:2303.17678}.

\bibitem[CGR06]{CGR}
V.~Chernousov, P.~Gille, and Z.~Reichstein.
\newblock Resolving {$G$}-torsors by abelian base extensions.
\newblock {\em J. Algebra}, 296(2):561--581, 2006.

\bibitem[CSZ23]{cheltsov}
I.~Cheltsov, A.~Sarikyan, and Z.~Zhuang.
\newblock Birational rigidity and alpha invariants of {F}ano varieties, 2023.
\newblock {\tt arxiv:2304.11333}.

\bibitem[CTS87]{Colliot-S}
J.-L. Colliot-Th\'{e}l\`ene and J.-J. Sansuc.
\newblock La descente sur les vari\'{e}t\'{e}s rationnelles. {II}.
\newblock {\em Duke Math. J.}, 54(2):375--492, 1987.

\bibitem[CTS07]{CT-inv}
J.-L. Colliot-Th\'{e}l\`ene and J.-J. Sansuc.
\newblock The rationality problem for fields of invariants under linear
  algebraic groups (with special regards to the {B}rauer group).
\newblock In {\em Algebraic groups and homogeneous spaces}, volume~19 of {\em
  Tata Inst. Fund. Res. Stud. Math.}, pages 113--186. Tata Inst. Fund. Res.,
  Mumbai, 2007.

\bibitem[DI09]{DI}
I.~V. Dolgachev and V.~A. Iskovskikh.
\newblock Finite subgroups of the plane {C}remona group.
\newblock In {\em Algebra, arithmetic, and geometry: in honor of {Y}u. {I}.
  {M}anin. {V}ol. {I}}, volume 269 of {\em Progr. Math.}, pages 443--548.
  Birkh\"{a}user Boston, Boston, MA, 2009.

\bibitem[DR15]{DR}
A.~Duncan and Z.~Reichstein.
\newblock Versality of algebraic group actions and rational points on twisted
  varieties.
\newblock {\em J. Algebr. Geom.}, 24(3):499--530, 2015.

\bibitem[HKT21]{HKTsmall}
B.~Hassett, A.~Kresch, and Yu. Tschinkel.
\newblock Symbols and equivariant birational geometry in small dimensions.
\newblock In {\em Rationality of Varieties}, volume 342 of {\em Progr. Math.},
  pages 201--236. Birkh\"auser, Cham, 2021.

\bibitem[HT23]{HTNagoya}
B.~Hassett and Yu. Tschinkel.
\newblock Torsors and stable equivariant birational geometry.
\newblock {\em Nagoya Math. J.}, 250:275--297, 2023.

\bibitem[Isk08]{Isk-s3}
V.~A. Iskovskikh.
\newblock Two non-conjugate embeddings of {$S_3\times Z_2$} into the {C}remona
  group. {II}.
\newblock In {\em Algebraic geometry in {E}ast {A}sia---{H}anoi 2005},
  volume~50 of {\em Adv. Stud. Pure Math.}, pages 251--267. Math. Soc. Japan,
  Tokyo, 2008.

\bibitem[Kat92]{Katsylo}
P.~Katsylo.
\newblock On the birational classification of linear representations, 1992.
\newblock MPI preprint.

\bibitem[Kol24]{Kollar}
J.~Koll\'ar.
\newblock Automorphisms and twisted forms of rings of invariants.
\newblock {\em Milan J. Math.}, 2024.

\bibitem[KT22]{BnG}
A.~Kresch and Yu. Tschinkel.
\newblock Equivariant birational types and {B}urnside volume.
\newblock {\em Ann. Sc. Norm. Super. Pisa Cl. Sci. (5)}, 23(2):1013--1052,
  2022.

\bibitem[LPR06]{lemire}
N.~Lemire, V.~L. Popov, and Z.~Reichstein.
\newblock Cayley groups.
\newblock {\em J. Amer. Math. Soc.}, 19(4):921--967, 2006.

\bibitem[NP23]{NP}
G.~Nebe and R.~Parker.
\newblock Equivariant quadratic forms in characteristic 2.
\newblock {\em Bull. Lond. Math. Soc.}, 55(2):668--679, 2023.

\bibitem[Ser58]{Serre1958}
J.-P. Serre.
\newblock Espaces fibrés algébriques.
\newblock {\em Séminaire Claude Chevalley}, 3:1--37, 1958.

\bibitem[Tot09]{Totaro}
B.~Totaro.
\newblock Birational geometry of quadrics.
\newblock {\em Bull. Soc. Math. France}, 137(2):253--276, 2009.

\bibitem[TYZ24]{TYZ}
Yu. Tschinkel, K.~Yang, and Zh. Zhang.
\newblock Equivariant birational geometry of linear actions.
\newblock {\em EMS Surv. Math. Sci.}, 11(2):235--276, 2024.

\end{thebibliography}

\end{document}